\documentclass[12pt]{article}
\usepackage[margin=3 cm]{geometry} 
\usepackage[USenglish]{babel}
\usepackage{amsmath,latexsym,amssymb,amsfonts,amssymb,graphicx,mathrsfs}

\usepackage{multicol}
\usepackage{pdfpages}
\usepackage{times}
\usepackage{graphicx}    
\usepackage{newcent}
\usepackage{amsthm}
\usepackage{amsxtra}
\usepackage{amstext}
\usepackage{authblk}
\usepackage{hyperref}
\usepackage{color}
\usepackage{verbatim}
\usepackage{chngcntr}
\usepackage{marginnote}
\usepackage[all]{xy}

\usepackage{tikz,pgfplots}
\usetikzlibrary{matrix}
\usepackage{pgfplotstable}
\pgfplotsset{compat=1.3,compat/path replacement=1.5.1}
\usepgfplotslibrary{external}
\usepgfplotslibrary{fillbetween} 
\usetikzlibrary{automata, arrows}
\usetikzlibrary{decorations.markings}
\tikzset{
  set arrow inside/.code={\pgfqkeys{/tikz/arrow inside}{#1}},
  set arrow inside={end/.initial=>, opt/.initial=},
  /pgf/decoration/Mark/.style={
    mark/.expanded=at position #1 with
    {
      \noexpand\arrow[\pgfkeysvalueof{/tikz/arrow inside/opt}]{\pgfkeysvalueof{/tikz/arrow inside/end}}
    }
  },
  arrow inside/.style 2 args={
    set arrow inside={#1},
    postaction={
      decorate,decoration={
        markings,Mark/.list={#2}
      }
    }
  },
}

\tikzset{partial ellipse/.style args={#1:#2:#3}{insert path={+ (#1:#3) arc (#1:#2:#3)}}}
\usepackage{caption}
\usepackage{subcaption}

\DeclareMathOperator\Arg{Arg}
\DeclareMathOperator\graphh{graph}

\newtheorem{remark}{Remark}

\newtheorem{definition}{Definition}
\newtheorem{theorem}{Theorem}[section]
\newtheorem{corollary}{Corollary}[section]
\newtheorem{lemma}{Lemma}[section]

\newtheorem*{remark*}{Remark}

\newcommand{\abs}[1]{\left\vert#1\right\vert}
\newcommand{\norm}[1]{\left\Vert#1\right\Vert}
\newcommand{\RR}{\mathbb R}
\newcommand{\ZZ}{\mathbb Z}

\newcommand{\TT}{\mathbb T}

\begin{document}
\title{Twist dynamics and Aubry-Mather sets around a periodically perturbed point-vortex \footnote{S.M. has been partially supported by
the project 'Geometric and numerical analysis of dynamical systems and applications to mathematical physics'
(MTM2016-76702-P), and the 'Juan de la Cierva-Formaci\'on' Programme (FJCI-2015-24917). V.O. has been partially supported by the  Spanish MINECO (FPI grant No. BES-2015-074096) and by the project MTM2017-82348-C2-1-P 
. }}

\author[1]{Stefano Mar\`o}
\author[2]{V\'\i ctor Ortega}
\affil[1]{Dipartimento di Matematica,  Universit\`a di Pisa, \protect\\ 56127 Pisa, Italy. \protect\\ E-mail: stefano.maro@unipi.it }
\affil[2]{Departamento de Matem\'atica Aplicada, Universidad de Granada, \protect\\ 18071 Granada, Spain. \protect\\ E-mail: victortega@ugr.es}

\renewcommand\Affilfont{\small}

\date{}
\maketitle

\begin{abstract}
We consider the model of a point-vortex under a periodic perturbation and give sufficient conditions for the existence of generalized quasi-periodic solutions with rotation number. The proof uses Aubry-Mather theory to obtain the existence of a family of minimal orbits of the Poincar\'e map associated to the system.
\end{abstract}


\section{Introduction}
We study the advection of a passive particle in a two-dimensional ideal fluid. This phenomena can be described by the Lagrangian version of fluid mechanics: the particle moves according to a Hamiltonian system with the streamfunction playing the role of the Hamiltonian.

Given the vorticity $\omega$ of the incompressible fluid, the streamfunction is defined as a solution of the Poisson equation $-\Delta\Psi =\omega$. From a physical point of view, a vortex is a zone of high vorticity. Mathematically, a vortex in the plane can be defined in different ways. It can be defined as a singularity of the vorticity or through a compact set of finite vorticity (vortex patch).  See \cite{Aref1,HHFH,Kozlov,MaWe,Newton} for a summary on the various definitions.

We will be concerned with a point-vortex, defined as a Dirac delta of the vorticity. Under this definition, the streamfunction is the fundamental solution of the $2$-dimensional Laplacian. These ideas were introduced in the seminal works of Helmholzt and Kirchoff in the
XIXth century and nowadays point-vortices are studied as a branch of Fluid Mechanics with deep connections with Celestial Mechanics and Hamiltonian systems.

A point-vortex induces a streamfunction $\Psi_0=\frac{\Gamma}{4\pi}\ln (x^2+y^2)$ where $\Gamma$ is the circulation (or strength) of the vortex and, up to a rescaling of unit, it can be set $\Gamma=2\pi$. The solution of the corresponding Hamiltonian system are circular path around the orgin and describe the trajectories of a passive particle under the influence of the vortex. The frequency of rotation is inversely proportional to the radius of the path and tends to infinity as the radius tends to zero.

We will study how this integrable dynamics is affected by the superposition of an external periodic time dependent streamfunction $p(t,x,y)$. 
More precisely, we consider the Hamiltonian
\begin{equation}\label{Ham}
\Psi(t,x,y) = \frac{1}{2}\ln(x^2+y^2) + p(t,x,y),
\end{equation}
and the associated Hamiltonian system 
\begin{equation}\label{uno}
\left\lbrace \begin{array}{lr}  \dot{x}= \partial_y \Psi (t,x,y)  &  \\ & \qquad(x,y)\in \, \mathcal{U}\setminus \left\lbrace 0 \right\rbrace,  \\  \dot{y}=- \partial_x \Psi (t,x,y)  &\end{array} \right.
\end{equation}
defined in a neighborhood $\mathcal{U}$ of the origin.

Physically, system \eqref{uno} can be interpreted to model the advection of a particle under the action of a steady vortex placed at the origin and a periodic time dependent background flow.

The dynamics of advected particles in non-stationary flows including vortices have been intensively studied from different perspectives \cite{BBPV,BCF,KTH,PC,RLW,TDH}. In particular, numerical tests suggest the presence of complex dynamics.

From an analytical point of view, in \cite{OOT} the authors studied the stability properties of the vortex. More precisely, they proved that if the external streamfunction is analytic, then KAM theory applies and invariant curves of the Poincar\'e map exist close to the singularity. As a byproduct, there exist quasiperiodic solutions  of the Hamiltonian systems with (sufficiently large) Diophantine frequencies. These can be seen as a reminiscent of the trajectories of the unperturbed Hamiltonian with Diophantine frequency.

In this paper, we will prove that, close to the singularity, quasi-periodic solutions exist for all frequency sufficiently large. Actually, our solutions will be a generalization of standard quasi-periodic solutions and in case of commensurable frequencies, we will get periodic solutions. These solutions exist also when KAM theory cannot be applied. Indeed, we will require very low regularity that prevents standard KAM theory from being applied.

To prove our result, we will apply a suitable version of Aubry-Mather theory \cite{Aubry,Mathertop} to the Poincar\'e map of system \eqref{uno}.  A similar scheme have been used to describe the dynamics of different systems \cite{Maro1,Maro3,Ortega96,shi,wang}.

For each sufficiently large real number $\alpha$, we will prove the existence of an invariant set $\mathcal{M}_\alpha$ (called Aubry-Mather set) with very interesting dynamical properties, among them each orbit in $\mathcal{M}_\alpha$ has rotation number $\alpha$. For irrational rotation numbers, the corresponding Aubry-Mather sets are either curves or a Cantor sets. Solutions of system \eqref{uno} with initial conditions in this set will be our generalized-quasi periodic solutions. In the rational case, the Aubry-Mather sets contain a periodic orbit.

In suitable variables, the Poincar\'e map will be an exact symplectic twist map of the cylinder. However, it will not be defined on the whole cylinder. Hence we cannot apply directly the result of Mather and we will prove an adapted version to this situation.

To apply our theorem, we will need to prove that the Poincar\'e map is exact symplectic and twist. The first property comes from the Hamiltonian character of the system. The twist condition is more delicate and relies on the behavior of the variational equation. We will give a proof following a perturbative approach. Here, we will ask that the perturbation has the origin as a zero of order $4$.

From the point of view of dynamics of symplectic diffeomorphisms, we will describe some aspects of the dynamics around a singularity. In the integrable case, the flow can be continuously extended to the singularity, defining it as a fixed point. However, this extension is not $\mathcal{C}^1$. In the perturbed case, in general is not even possible to guarantee continuity of this extension. Since the flow is not regular, all the results coming from the theory of elliptic fixed points and transformation to Birkhoff normal form cannot be applied directly.
We will overcome the problem of the singularity performing a change of variable that sends the singularity at infinity and has a regularizing effect. At this stage, the assumption of having the zero of order $4$ in the perturbation play a fundamental role.

The paper is organized as follows. In Section \ref{sec:state}  we state the problem and the main result. The definition of generalized quasi-periodic solution will be given in this section. In Section \ref{sec:preest} we introduce the regularizing variables and the Poincar\'e map together with some preliminary estimates. In Section \ref{sec:AM} we state and prove the suitable version of the Aubry-Mather theorem. In Section \ref{sec:exact} it is proved the property of exact symplectic and Section \ref{sec:twist} is dedicated to the proof of the twist property. The proof of the main result will be given in Section \ref{sec:proof}. Finally, we draw some conclusions in Section \ref{sec:conclu}. Some technical lemmas are relegated to the Appendix.

\section{Statement of the problem and main result}
\label{sec:state}

Let us consider the perturbed Hamiltonian system given by \eqref{Ham}-\eqref{uno}.
%
%
We suppose that the perturbation $p(t,x,y)$ belongs to the following class

\begin{definition}\label{defr}
   Given $\varepsilon > 0$, consider the open disk around the origin  $\mathbb{D}_{\varepsilon} = \left\lbrace (x,y)\in \mathbb{R}^2: x^2+y^2 < \varepsilon^2 \right\rbrace$. We say that a continuous functions $f : \RR \times \mathbb{D}_{\varepsilon}\longrightarrow  \RR$ belongs to the class $\mathcal{R}^k_{\varepsilon}$ if 

 \begin{itemize}
 \item [\textbf{i)}]$f(t+1,x,y) = f(t,x,y)$.
 \item [\textbf{ii)}]$f \in \mathcal{C}^{0,k}(\RR \times \mathbb{D}_{\varepsilon})$ i.e. $f$ is $\mathcal{C}^k$ w.r.t the spatial variables $(x,y)$ and all the partial derivatives are continuous w.r.t. $(t,x,y)$.
  
 \end{itemize}
\end{definition}

%

Now, given any $N\in\mathbb{N}$  we give the notion of \textit{zero of order $N$} of a function $f\in\mathcal{R}^k_{\varepsilon}$. The following definition will be of particular interest in the case $N>k$. 
\begin{definition}
  \label{deforder}
   Given a function $f \in \mathcal{R}^k_{\varepsilon}$ we say that the origin is a \textit{zero of order $N$} if there exist $T_N, \tilde{f}  \in \mathcal{C}^{0,k}(\RR \times \mathbb{D}_{\varepsilon})$ such that,
    \[ f(t,x,y) = T_N(t,x,y)+\tilde{f}(t,x,y)
    \]
and satisfying the following properties. 
 \begin{itemize}
 	\item  $T_N(t,x,y) =  \underset{i+j=N}{\sum} \alpha_{i,j}(t)x^i y^j$  is a homogeneous polynomial of degree $N$ with $\mathcal{C}^1$ coefficients,
 	
 	\item there exists a constant $C$ such that, for all $(t,x,y)\in\RR\times\mathbb{D}_\varepsilon$,   
     \begin{align*}
    & |\tilde{f}(t,x,y)| \leq C(|x|^{N+1}+|y|^{N+1}),\\ &
|\partial^{(m)} \tilde{f}(t,x,y)| \leq C(|x|^{N-m+1}+|y|^{N-m+1})\quad \mbox{for} \quad 1\leq m\leq k.
	\end{align*}
    \end{itemize}
\end{definition}

 Our result gives the existence of particular families of solutions: periodic and quasi-periodic solutions in a generalized sense. To define them, given a solution $(x(t),y(t))$ of \eqref{uno}, consider the functions
\begin{equation}\label{defrt}
r(t)=\frac{1}{2(x(t)^2+y(t)^2)},
 \quad \theta(t) = -\Arg [x(t)+i y(t)],
 \end{equation}
having a relation with the standard polar coordinates. Actually $\theta(t)$  represents the angle in the clockwise sense, while $r(t)$ is, up to a scaling constant, the inverse of the square of the radius.

\begin{definition}
 We say that the solution $(x(t),y(t))$, defined for $t\in\mathbb{R}$
    \begin{itemize}
    \item is \textit{non-singular }if
      \[
      \sup_{t\in\RR}r(t)<\infty;
      \]
    \item is \textit{bounded} if there exists $A>0$ such that 
      \[
      \inf_{t\in\RR}r(t)>A;
      \]
      \item has \textit{monotone argument} if $\theta(t)$ is monotone;
  \item has \textit{rotation number} $\alpha$ if
    \[
   \frac{1}{2\pi} \lim_{t\to\infty}\frac{\theta(t)}{t}=\alpha.
    \]
    \end{itemize}
\end{definition}

\begin{remark} \em{
A non-singular bounded solution with monotone argument rotates clockwise in a closed annulus around the origin. Moreover, the rotation number represents the average angular velocity. 
}
\end{remark}

We are ready to state the main result.

\begin{theorem}\label{Theo}
 Suppose that $p\in\mathcal{R}^3_{\varepsilon}$ is such that the origin is a zero of order $4$.

  Then there exists $\bar{\alpha}$ sufficiently large such that for every $\alpha>\bar{\alpha}$ 
  there exist a family of non-singular, bounded solutions 
  \[
  \left\{ (x(t),y(t))_\xi\right\}_{\xi\in\RR}
  \]
  with monotone argument and rotation number $\alpha$. These solutions are such that the related functions $r(t),\theta(t)$ defined in \eqref{defrt} satisfy, for every $t,\xi\in\RR$,
\begin{align}\label{contheo}
 & (r(t),\theta(t))_{\xi+2\pi} =(r(t),\theta(t))_\xi+(0,2\pi)  \\ &
%
\label{contheo2}
(r(t+1),\theta(t+1))_{\xi}=(r(t),\theta(t))_{\xi+2\pi\alpha}.
\end{align}
\end{theorem}   

\begin{remark}\label{rem2} \em{
The possible crossings of each solution occur in a determined way. Actually, if $(x(t_1),y(t_1))_\xi = (x(t_2),y(t_2))_\xi$ for $t_1 -t_2 \notin \ZZ$, since $1$ is the minimal period of the perturbation, $(\dot{x}(t_1),\dot{y}(t_1))_\xi \neq (\dot{x}(t_2),\dot{y}(t_2))_\xi$.

    If $\alpha=s/q\in\mathbb{Q}$, then the solutions satisfy
    \[
 (r(t+q),\theta(t+q))_{\xi} =(r(t),\theta(t))_\xi+(0,2\pi s)
    \]
    and are said $(s,q)$-{\it periodic}. These solutions make $s$ revolutions around the singularity in time $q$. 
If $\alpha\in\RR\setminus\mathbb{Q}$, solutions satisfying \eqref{contheo}-\eqref{contheo2} can be seen as generalized quasi periodic. Actually, consider the function
  \[
  \Phi_\xi(a,b)= (r(a),\theta(a))_{b-2\pi\alpha a+\xi}.
  \]
  This function is doubly-periodic in the sense that
  \begin{align*}
    \Phi_\xi(a+1,b)&= (r(a+1),\theta(a+1))_{b-2\pi\alpha a+\xi-2\pi\alpha}=\Phi_\xi(a,b), \\
    \Phi_\xi(a,b+2\pi)&=(r(a),\theta(a))_{b-2\pi\alpha a+\xi+2\pi} =\Phi_\xi(a,b)+(0,2\pi).
  \end{align*}
  and
  $\Phi_\xi(t,2\pi\alpha t)=(r(t),\theta(t))_{\xi}$.
  If the function $\xi\mapsto\Phi_\xi$ is continuous, then these solutions are classical quasi-periodic solutions with frequencies $(1,\alpha)$ in the sense of \cite{SM} (see also \cite{OrtegaLisboa}). We will not guarantee the continuity, however, the function $\xi\mapsto\Phi_\xi$ will have at most jump discontinuities and if $\xi$ is a point of continuity then so are $\xi+2\pi\alpha,\xi+2\pi$. Finally, the set $Cl\{ (x(0),y(0))_\xi : \xi\in \RR \}$ is either a curve or a Cantor set, recovering the classical definition of quasi-periodic solution in the case of having an invariant curve. 
  }
\end{remark}

\section{Some estimates on the solutions and the Poincar\'e map } \label{sec:preest}

Let us consider system \eqref{uno} and, following section $4.1$ of \cite{OOT}, consider the change of variables $(x,y)=\varphi(\theta,r)$
defined by
\[
x = \frac{\cos\theta}{\sqrt{2r}}, \qquad y = -\frac{\sin\theta}{\sqrt{2r}}.
\]
These variables comes from applying first the Kelvin transform and subsequently the change to symplectic polar coordinates.
System \eqref{uno} transforms into
\begin{equation} \label{tres}
  \left\lbrace \begin{array}{l}
    \dot{r}= 4r^2 \,  \partial_{\theta} H (t,r,\theta)
    \\  \dot{\theta}=- 4r^2 \, \partial_r H (t,r,\theta)
  \end{array} \right.
\end{equation}
where $ H (t,r,\theta)= - \frac{1}{2}\ln(2r) + h (t,r,\theta)\: $ and $\: h (t,r,\theta)= p\left(t, \frac{\cos\theta}{\sqrt{2r}},-\frac{\sin\theta}{\sqrt{2r}} \right)$.
System \eqref{tres} is still a periodic planar Hamiltonian system with symplectic form $\tilde{\lambda}=\frac{1}{4r^2}\,\mathrm{d}r \wedge \mathrm{d}\theta$.
Moreover, the change of variables $\varphi$ transforms the domain $\RR\times\mathbb{D}_{\varepsilon}$ into the domain $\RR\times\mathcal{D}$ with
\begin{equation*}
\mathcal{D} = \left\lbrace (r,\theta) \in ]r_\ast,\infty[\times\TT \: : \: r_\ast=\frac{1}{2\varepsilon^2} \right\rbrace. 
\end{equation*}



\noindent Let us write the Cauchy problem associated to system \eqref{tres}, in the following form:

\begin{equation} \label{cuatro}
\left\lbrace \begin{array}{l}  \dot{r}= F(t,r,\theta) , \\  \dot{\theta}=2r + G(t,r,\theta) , \\ (r(0), \theta(0))=(r_0,\theta_0) \end{array} \right.
\end{equation}
where,
\begin{align}
  \nonumber
  F(t,r, \theta) &=4r^2 \partial_{\theta} \left[p\left(t, \frac{\cos \theta}{\sqrt{2r}}, \frac{-\sin \theta}{\sqrt{2r}}\right)\right],
  \\
  \label{defg}
G(t,r, \theta) &= -4r^2 \partial_{r} \left[p\left(t, \frac{\cos \theta}{\sqrt{2r}}, \frac{-\sin \theta}{\sqrt{2r}}\right)\right].
\end{align}
 Since $p\in\mathcal{R}^3_\varepsilon$, the vector field in \eqref{cuatro} is continuous and $\mathcal{C}^2$ in the spatial variables.
This guarantees existence and uniqueness of the solution.
%
\begin{remark}
\em{
  The change of variables $\varphi$ has the effect to transform the phase space from the plane to the cylinder. The singularity is moved from the origin to $r\rightarrow\infty$. In this sense, the change of variables has a regularizing effect since the functions $F,G$ in \eqref{cuatro} are bounded for $r\rightarrow\infty$. The fact that the origin is a zero of order $4$ plays a fundamental role in this discussion. See estimate \eqref{bound33} in the following Lemma \ref{evosol} for more details.
  }
\end{remark}

Since the domain $\mathcal{D}$ is not invariant, we need to control the growth of the solutions. For this purpose, given $a>r_*$ we introduce the set 
\[
\Sigma (a) = ]a,\infty[\times\TT \subset \mathcal{D}
    \]
%
and prove the following lemma, whose meaning is illustrated in Fig. \ref{domains}.

\begin{figure}
\centering \hspace*{-0.1cm}
  \begin{tikzpicture}[xscale=2.4,yscale=1.8]

\draw[line width=0.43mm,dashed]   (0,-0.6) --(0,0);

\draw[line width=0.43mm,dashed] (1.5,-0.6) -- (1.5,0);

\draw[line width=0.45mm,-latex]   (0,0)  node [left] at (0,0) {\scalebox{0.9}[0.7]{$r_\ast$}}  -- (0,2) node[anchor=west] {$\emph{r}$};

\draw[line width=0.45mm] (1.5,0) -- (1.5,2);

\draw[thick, -latex] (0.75,0) [partial ellipse=-180:-110:0.75cm and 0.15cm]
node[anchor=north ] {\emph{$\theta$}};
\draw[thick] (0.75,0) [partial ellipse=-115:0:0.75cm and 0.15cm];
\draw[thick,dashed] (0.75,0) [partial ellipse=0:180:0.75cm and 0.15cm];


\draw[rotate=0,cyan,line width=0.4mm] (0.75,0.9) [partial ellipse=-180:0:0.75cm and 0.15cm];
\draw[rotate=0,cyan,line width=0.4mm,dashed] (0.75,0.9) [partial ellipse=0:180:0.75cm and 0.15cm]
node [left,cyan] at (0,0.9) {\scalebox{0.7}[0.7]{$a_\ast$}};

%
%

\draw[thick,->,red] (0.19,1.3) 
 to [out=45,in=160] (0.75,1.45);
\draw[thick,red] (0.75,1.45) to [out=-30,in=60] (0.47,1.05) to [out=240,in=120]  (0.9,0.12) 
to [out=-30,in=210]  (1.25,0.3);

 \fill[red] (0.19,1.3) circle[radius=0.3mm]
 node[below ,red] {\scalebox{0.7}[0.7]{$(r_0,\theta_0)_1$}};

 \fill[red] (1.25,0.3) circle[radius=0.3mm]
 node[above,red] {\scalebox{0.6}[0.6]{$(r(1),\theta(1))_1$}};

\draw[thick,->,orange] (0.8,0.65) 
 to [out=45,in=170] (1.2,1);
\draw[thick,orange] (1.2,1) to
[out=-30,in=180] (0.27,-0.4) to [out=0,in=-145]  (0.3,0.3);

 \fill[orange] (0.8,0.65) circle[radius=0.3mm]
 node[below,orange] {\scalebox{0.7}[0.7]{$(r_0,\theta_0)_2$}};

 \fill[orange] (0.3,0.3) circle[radius=0.3mm]
 node[ above,orange] {\scalebox{0.6}[0.6]{$(r(1),\theta(1))_2$}};

\end{tikzpicture}
  \caption{
Domains and evolution of two solutions over a period. $(r(t),\theta(t))_1$ represents the solution with $(r_0,\theta_0)_1\in\Sigma(a_\ast)$ and $(r(t),\theta(t))_2$ represents the solution with $(r_0,\theta_0)_2\notin \Sigma(a_\ast)$. Note that the solution $(r(t),\theta(t))_1$ remains in the domain $\mathcal{D}$. 
  }
\label{domains}
\end{figure}
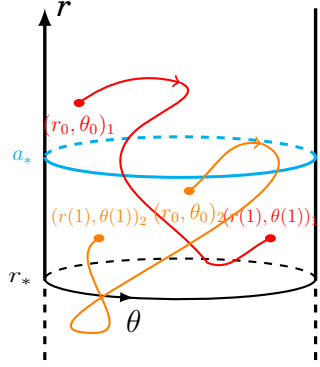

\begin{lemma}\label{evosol}

  Let us assume that the origin is a zero of order $4$ of the function  $p\in \mathcal{R}^3_{\varepsilon}$.
 Then there exists $a_\ast > r_\ast$ such that if $(r_0, \theta_0)\in \Sigma(a_\ast)$, the corresponding solution of \eqref{cuatro} is well defined on $t\in [0,1]$ and $ (r(t),\theta(t)) \in \mathcal{D} $ for all $t\in [0,1]$.
Moreover, the following estimate holds
\begin{equation}\label{lemma1}
\abs{r(t)- r_0} + \abs{\theta(t) - \theta_0 - 2r_0t} \leq K \:\quad \textnormal{if} \:\quad t\in [0,1]
\end{equation}
for some $K>0$.

\end{lemma}

\begin{proof}

Since the origin is a zero of order $4$ of $p$, there exists a constant $C>0$ such that

\begin{equation*}
\abs{\partial^{(1)} p \, (t,x,y)} \leq C ( \abs{x}^3 + \abs{y}^3) \:\quad  \mbox{in }  \mathbb{R}\times\mathbb{D}_\varepsilon.
\end{equation*}
Then, from the definition of $F$and $G$, we have 
\begin{equation}\label{bound33}
\abs{F(t,r,\theta)} + \abs{2r}\abs{G(t,r,\theta)} \leq C_1
\end{equation} for any $t\in \RR$ and $(r,\theta) \in \mathcal{D}$.

We shall prove that $a_*=r_\ast+C_1$ satisfies the lemma. Fix $(\theta_0,r_0)\in \Sigma(a_*)$ and consider the corresponding solution $(\theta(t),r(t))$. By continuity there exists $\tau$ such that $r(t)$ is well defined and $r(t)>r_\ast$ for $t\in[0,\tau]$. Suppose that $\tau<1$ otherwise we are done. Integrating the first equation of \eqref{cuatro} and using \eqref{bound33} we have 
\[ \abs{r(t)-r_0} \leq C_1t \: \quad \textnormal{if} \: \quad t\in [0,\tau].\]
In particular, $r(\tau) \geq r_0 - C_1\tau >r_\ast$. Hence we can continue the solution until time $\tau+\tau_1$. Suppose that $\tau+\tau_1<1$ otherwise we are done. Hence, as before $r(\tau+\tau_1) \geq r_0 - C_1(\tau+\tau_1) >r_\ast$. Repeating this procedure we can reach   $\tau =1$.     

Finally, integrating the second equation of \eqref{cuatro}, we deduce that
\begin{equation*}
\abs{\theta (t) - \theta_0 - 2r_0t} \leq 2C_1 + \frac{C_1}{2(r_0-C_1)}.
\end{equation*}
Here we have employed \eqref{bound33} and the above estimates on $r(t)$.

\end{proof}

Now, let us introduce the Poincar\'e map $\mathcal{P}$ as
    \[ \begin{array}{rcl} \mathcal{P}: \quad \Sigma(a_\ast) =  \left]a_\ast,\infty\right[ \times \mathbb{T} & \longrightarrow & \mathcal{D}\subset\RR\times\TT \\ (r_0,\theta_0) & \longmapsto & (r_1,\theta_1) = \left(r(1;r_0,\theta_0), \theta (1; r_0,\theta_0)\right)\end{array}
    \]
    where $(r(t;r_0,\theta_0), \theta (t; r_0,\theta_0))$ is the solution with initial condition $(r(0),\theta (0))=(r_0,\theta_0)$. Lemma \ref{evosol} together with existence and uniqueness of the solutions of problem \eqref{cuatro} guarantee that the Poincar\'e map is well defined.\\
Due to the regularity of the vector field of \eqref{tres}, $\mathcal{P}\in\mathcal{C}^2(\Sigma(a_\ast))$, concretely is a diffeomorphism of a section of the cylinder.

The proof of the theorem will be a consequence of a suitable version of the so called Aubry-Mather theory applied to the previous Poincar\'e map. The following section is dedicated to the statement and proof of this result.

\section{A generalized Aubry-Mather theorem}
\label{sec:AM}

\noindent We denote by $\mathfrak{C} = \RR\times\TT$, $\mathbb{T}=\RR/2\pi\mathbb{Z}$ the cylinder and consider the strip $\Sigma:=(a,b)\times \mathbb{T}$ and the corresponding lift $\tilde{\Sigma}:=(a,b)\times \RR$. 

\noindent Consider a $\mathcal{C}^2$ diffeomorphism
\[\begin{array}{rcl} \Phi: \quad\Sigma & \longrightarrow & \mathfrak{C} \\ (r,\theta) & \longmapsto & (r_1,\theta_1)= \left(\mathcal{F}(r,\theta),\mathcal{G}(r,\theta)\right). \end{array}\]
We denote the lift by 
\begin{equation}
  \label{philift}
  \begin{array}{rcl} \Phi: \quad\tilde{\Sigma} & \longrightarrow & \RR^2 \\ (r,x) & \longmapsto & (r_1,x_1)= \left(\mathcal{F}(r,x),\mathcal{G}(r,x)\right) \end{array}
  \end{equation}
where
\begin{align*}
  \mathcal{F}(r,x+ 2\pi) &= \mathcal{F}(r,x), \\
  \mathcal{G}(r,x+ 2\pi) &= \mathcal{G}(r,x) + 2\pi .
\end{align*}

\noindent Consider a $\mathcal{C}^2$  function with Lipschitz inverse 

\[ \begin{array}{rcl} f: (a,b) & \longrightarrow & \RR \\ r & \longmapsto & f(r), \end{array}\]
such that $f'$ never vanishes. Without loss of generality we fix $f'>0$.

We suppose that $\Phi$ is exact symplectic with respect to the form
\[\tilde{\lambda} = \mathrm{d}f(r) \wedge \mathrm{d}\theta = f'(r) \mathrm{d}r \wedge \mathrm{d}\theta\]
that is, there exists a $\mathcal{C}^2$ function 
\[\begin{array}{rcl} \mathcal{S}:\quad \Sigma & \longrightarrow & \RR \\ (r,\theta) & \longmapsto & \mathcal{S}(r,\theta) \end{array}\] such that 
\[
\mathrm{d}\mathcal{S}(r,\theta)= f(r_1)\,\mathrm{d}\theta_1 - f(r)\,\mathrm{d}\theta, \quad \forall(r,\theta)\in \Sigma.\]
  \begin{remark}
  \em{
  Note that the function $\mathcal{S}(r,\theta)$ is defined in the cylinder, hence the lift $\mathcal{S}(r,x)$ must be a $2\pi$-periodic function in the variable $x$ such that
  \begin{equation}
    \label{exact}
    \mathcal{S}_r(r,x) = f(\mathcal{F}(r,x))\mathcal{G}_r(r,x), \qquad
    \mathcal{S}_x(r,x) = f(\mathcal{F}(r,x))\mathcal{G}_x(r,x) -f(r).
    \end{equation}
    }
 \end{remark}

\noindent We also suppose that $\Phi$ is twist, that is
\begin{equation}
  \label{deftwist}
  \partial_r \mathcal{G}(r,\theta) > 0  \qquad \forall (r,\theta)\in \Sigma.
  \end{equation}

\noindent Suppose additionally that the following uniform limits (w.r.t. $x$) exist
\begin{align*}
\alpha^+(x)&:= \frac{1}{2\pi}\left(\lim_{r \to b} \mathcal{G}(r,x)- x\right), \\ 
\alpha^-(x)&:= \frac{1}{2\pi}\left(\lim_{r \to a} \mathcal{G}(r,x)- x\right).
\end{align*}
Note that $\alpha^{\pm}(x)$ are $2\pi$-periodic $\mathcal{C}^2$ functions
 and define
\[
W^+=\min_{x}\alpha^+(x),\qquad W^-=\max_{x}\alpha^-(x).
\]

The main result of this section deals with the existence of special orbits of the diffeomorphism $\Phi$. To state the Theorem, we recall that a sequence $(x_n)_{n\in\mathbb{Z}}$ of real numbers is increasing if $x_n<x_{n+1}$ for all $n\in\mathbb{Z}$ and we say that any two translates are comparable if for any $(s,q)\in\mathbb{Z}^2$ only one of the following alternatives holds
  \[
  \overline{x}_{n+q}+2\pi s > \overline{x}_{n} \:\forall n, \quad \overline{x}_{n+q}+2\pi s = \overline{x}_{n}\:\forall n, \quad \overline{x}_{n+q}+2\pi s < \overline{x}_{n} \:\forall n.
  \]   
We are now ready to state the main result of this section:

\begin{theorem}\label{MatherTh}
With the previous setting, suppose that $W^+-W^->8\pi$ and fix 
$ \alpha$ such that $2\pi\alpha \in \left(W^- + 4\pi, W^+ - 4\pi\right)$. Then
\begin{itemize}
\item if $\alpha=s/q\in\mathbb{Q}$ there exists a $(s,q)$-periodic orbit $ ( \overline{r}_n, \overline{x}_n)_{n\in \mathbb{Z}}$ such that
    \[    \overline{r}_{n+q} =   \overline{r}_{n} ,\qquad \overline{x}_{n+q}= \overline{x}_{n}+2\pi s \quad \forall n\in\mathbb{Z};
    \]
  \item if $\alpha\in\RR\setminus\mathbb{Q}$ there exists a compact invariant subset $\mathcal{M}_\alpha\subset\Sigma$ (and a corresponding subset $\tilde{\mathcal{M}}_\alpha\subset\tilde{\Sigma}$ ) with the following properties:
    \begin{itemize}
    \item denoting $\pi:\Sigma\rightarrow \TT$ the projection, $\pi|_{\mathcal{M}_{\alpha}}$ is injective and $\mathcal{M}_\alpha = \graphh u$ for a Lipschitz function $u:\pi(\mathcal{M}_\alpha)\rightarrow\RR$,
        \item each orbit $ ( \overline{r}_n, \overline{x}_n)_{n\in \mathbb{Z}}\in \tilde{\mathcal{M}}_\alpha$ is such that the sequence $(\overline{x}_n)$ is increasing and any two translates are comparable, 
        \item each orbit $ ( \overline{r}_n, \overline{x}_n)_{n\in \mathbb{Z}}\in \tilde{\mathcal{M}}_\alpha$ has rotation number $\alpha$, i.e.
       \[
    \frac{1}{2\pi}\lim_{n \to \infty} \frac{\overline{x}_n}{n}=\alpha;
    \]    
              \item the set $\mathcal{M}_\alpha$ is either an invariant curve or a Cantor set.
      \end{itemize}

\end{itemize}

\end{theorem}

The following corollary gives an equivalent interpretation of the result and has been proven in \cite{Maro1}.
\begin{corollary}\label{MatherCor}
 For each $\alpha$ there exists two functions $\phi,\eta:\RR\rightarrow\RR$ such that, for every $\xi\in\mathbb{R}$
  \begin{align*}
    &\phi(\xi+2\pi) = \phi(\xi)+2\pi, \quad \eta(\xi+2\pi)=\eta(\xi),\\
    &\Phi(\phi(\xi),\eta(\xi))=(\phi(\xi+2\pi\alpha),\eta(\xi+2\pi\alpha))
  \end{align*}
  where $\phi$ is monotone (strictly if $\alpha\in\RR\setminus\mathbb{Q}$ ) and $\eta$ is of bounded variation.
\end{corollary}


%
The proof of Theorem \ref{MatherTh} will make use of the {\it generating function}. We introduce it in the following

\begin{lemma}\label{generating}
  There exists an open connected set $\mathcal{B}\subset\mathbb{R}^2$ and a function $h:\mathcal{B}\rightarrow \mathbb{R}$, called {\it generating function} such that
  \begin{itemize}
    \item[o)] $\mathcal{B}$ is invariant under the translation $(x,x_1)\mapsto (x+2\pi,x_1+2\pi)$;
  \item[i)] $h\in \mathcal{C}^3(\mathcal{B})$;
  \item[ii)] $h(x+2\pi,x_1+2\pi)=h(x,x_1)$ for all $(x,x_1)\in\mathcal{B}$;
  \item[iii)]   $\partial^2_{xx_1} h(x,x_1)<0$ for all $(x,x_1)\in\mathcal{B}$;
  \item[iv)]     a sequence $(\overline{r}_n,\overline{x}_n)_{n\in\mathbb{Z}}$ is an orbit of $\tilde\Phi$ iff for all $n\in \mathbb{Z}$ 
    \[
    \partial_{1}h(\overline{x}_n,\overline{x}_{n+1})+ \partial_{2}h(\overline{x}_{n-1},\overline{x}_{n})=0
    \quad\mbox{and}\quad
    f(\overline{r}_{n}) = - \partial_{1}h(\overline{x}_n,\overline{x}_{n+1}).
    \]
 
  \end{itemize}
  
\end{lemma}

\begin{figure}
\centering 
\begin{tikzpicture}
 	\begin{axis}[axis on top=true,xtick=\empty, ytick=				\empty, axis x line=center,
	    axis y line=center,
		minor tick num=1,
	    xlabel={$x$},
    	xmin=-1, xmax=10,
	    ylabel={$x_1$},
    	ymin=-1.8, ymax=9.4,
    	]
    \addplot[name path=A,red, samples=100, 						domain=-1:8]{0.3*sin(deg(6*x))+ 			4.3+x}
    	node[above left, scale=0.9, pos 			= .4,] {$\alpha^+(x) +  x$};
    \addplot[blue, dashed,samples=100, 			domain=-1:8]{4 +x} 
	    node[below right, sloped, 					scale=0.9, pos = .4] {$W^++x$};       	\addplot[blue, dashed,samples=100, 			domain=-1:8]{-1 +x} 
	    node[above left,sloped, 					scale=0.9, pos = .5] {$W^- + x				$};
    \addplot[name path=B,red, samples=100, 					domain=-1:8]{0.4*sin(deg(5*x))+ 			-1.4 + x}
        node[below right, scale=0.9, pos 		= .6] {$\alpha^-(x) + x$}
        ;
        \addplot[orange!18] fill between[of=A and B];
    \end{axis}
\draw node [scale=1.5] at (2,3){$\mathcal{B}$};
    
\end{tikzpicture}
\caption{Domain $\mathcal{B}$.}
\label{domainB}

\end{figure}
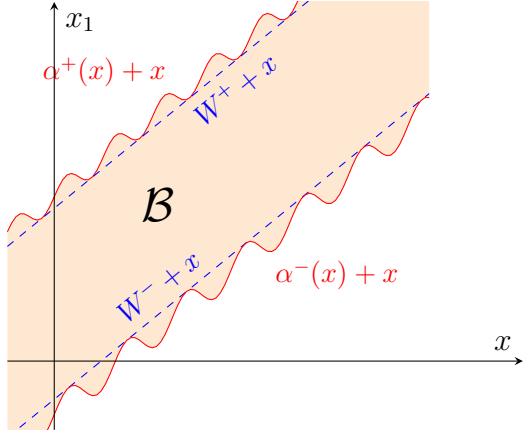   

\begin{proof}
By the twist property, $\alpha^+(x)>\alpha^-(x) \quad \forall x \in \RR$,
so that we can consider the open connected set (see figure \ref{domainB})
\[
\mathcal{B} = \left\lbrace(x,x_1)\in \RR^2:\alpha^-(x) < x_1-x <\alpha^+(x) \right\rbrace.
\]
From the periodic property of the functions $\alpha^{\pm}(x)$ this set is invariant under the  translation
\[
T_{2\pi,2\pi}(x,x_1) = (x+2\pi,x_1+2\pi).
\]
By the twist condition we can solve the implicit function problem
\[x_1=\mathcal{G}(r,x)\]
and obtain a unique $\mathcal{C}^2$ function $\mathcal{R}(x,x_1):\, \mathcal{B} \longrightarrow (a,b)$ such that
\[
x_1=\mathcal{G}(r,x)\Longleftrightarrow r=\mathcal{R}(x,x_1)
\]
and, by implicit differentiation,
\begin{equation}
  \label{implicit}
\mathcal{G}_r(\mathcal{R},x)\mathcal{R}_x+\mathcal{G}_x(\mathcal{R},x) =0, \qquad \mathcal{G}_r(\mathcal{R},x)\mathcal{R}_{x_1}=1.
\end{equation}
Moreover, uniqueness implies that $\mathcal{R}(x+2\pi,x_1+2\pi)=\mathcal{R}(x,x_1)$. Analogously we get
\[
r_1=\mathcal{F}(r,x)\Longleftrightarrow r_1 = \mathcal{F}(\mathcal{R}(x,x_1),x):=\mathcal{R}_1(x,x_1)
\]
with $\mathcal{R}_1(x+2\pi,x_1+2\pi)=\mathcal{R}_1(x,x_1) $.
Hence, the map \eqref{philift} is equivalent to
\begin{equation*}
  \left
  \lbrace
  \begin{array}{l}
    r_1=\mathcal{R}_1(x,x_1),
    \\
    r=\mathcal{R}(x,x_1)
  \end{array}
  \right.
  \quad \mbox{with } (x,x_1)\in \mathcal{B}. 
\end{equation*}

Now, we use the exact symplectic condition and define the generating function
\[
h(x,x_1):=\mathcal{S}(\mathcal{R}(x,x_1),x).
\]
This maps is clearly $\mathcal{C}^2(\mathcal{B})$, and, a posteriori, we will get $\mathcal{C}^3$ regularity. From the periodicity conditions of $\mathcal{S}$ and $\mathcal{R}$, one can prove the periodicity condition {\it ii)}.\\
To prove point {\it iii)}, we use (\ref{exact}),(\ref{implicit}) to get that  for all $(x,x_1)\in \mathcal{B}$,
\begin{align}
  \label{dxh}
\nonumber\partial_{x}h (x,x_1) &= \partial_{x} \mathcal{S}(\mathcal{R}(x,x_1),x) = \mathcal{S}_r(\mathcal{R}(x,x_1),x)\mathcal{R}_{x}+\mathcal{S}_x(\mathcal{R}(x,x_1),x) \\
  &= f(\mathcal{F}(\mathcal{R},x))\mathcal{G}_r(\mathcal{R},x)\mathcal{R}_{x} + f(\mathcal{F}(\mathcal{R},x))\mathcal{G}_x(\mathcal{R},x) -f(\mathcal{R})  \\\nonumber
&= -f(\mathcal{R}),
\end{align}
so that the twist condition and the monotonicity of $f$ imply
\begin{align*}
  \partial_{x,x_1}h(x,x_1)&= -\partial_{x_1}f(\mathcal{R}(x,x_1))=-f'(\mathcal{R})\partial_{x_1}\mathcal{R}(x,x_1)=-\frac{f'(\mathcal{R})}{\partial_r\mathcal{G}(x,x_1)}<0.
\end{align*}

To prove the last point, a similar computation as \eqref{dxh} gives  for all \\$(x,x_1)\in \mathcal{B}$,
\begin{align}
  \label{dx1h}
  \nonumber\partial_{x_1}h (x,x_1) &= \partial_{x_1} \mathcal{S}(\mathcal{R}(x,x_1),x) = \mathcal{S}_r(\mathcal{R}(x,x_1),x)\mathcal{R}_{x_1}     \\
 &=f(\mathcal{F}(\mathcal{R},x))\mathcal{G}_r(\mathcal{R},x)\mathcal{R}_{x_1} = f(\mathcal{F}(\mathcal{R},x))    \\ \nonumber&=f(\mathcal{R}_1).
  %
\end{align}
Equations \eqref{dxh}-\eqref{dx1h}, together with the regularity of $f,\mathcal{R},\mathcal{R}_1$ have the consequence that $h\in \mathcal{C}^3(\mathcal{B})$, proving point {\it i)}.

Less formally \eqref{dxh}-\eqref{dx1h} also imply that the map $\Phi$ can be expressed implicitly: 
\begin{equation*}
  \left
  \lbrace
  \begin{array}{l}
   \partial_{x_1}h (x,x_1)=f(r_1 )
    \\
    \partial_{x}h (x,x_1)=-f(r)
  \end{array}
  \right.
  \quad \mbox{with }(x,x_1)\in \mathcal{B}.
\end{equation*}
It means that an orbit  $(\overline{r}_n,\overline{x}_n)_{n\in\mathbb{Z}}$ is such that for every $n\in\mathbb{Z}$

\begin{equation*}
  \left
  \lbrace
  \begin{array}{l}
   f(\overline{r}_{n+1}) = \partial_{2}h(\overline{x}_n,\overline{x}_{n+1})
    \\
    f(\overline{r}_{n}) = - \partial_{1}h(\overline{x}_n,\overline{x}_{n+1}).
  \end{array}
  \right.
\end{equation*}
%
This implies $f(\overline{r}_{n}) = -\partial_{1}h(\overline{x}_n,\overline{x}_{n+1})= \partial_{2}h(\overline{x}_{n-1},\overline{x}_{n})$ so that 
\[\partial_{1}h(\overline{x}_n,\overline{x}_{n+1})+ \partial_{2}h(\overline{x}_{n-1},\overline{x}_{n})=0, \qquad \forall n\in \mathbb{Z}. \]

\end{proof}
\begin{remark}
\em{
  The equation
  \[
  \partial_{1}h({x}_n,{x}_{n+1})+ \partial_{2}h({x}_{n-1},{x}_{n})=0, \qquad \forall n\in \mathbb{Z}
  \]
  is known as \textit{discrete Euler-Lagrange equation}.
  }
  \end{remark}

The usual Mather's theorem (see Theorem \ref{Mather}), gives sufficient conditions on the generating function in order to get orbits with rotation number. In particular it is required $h\in \mathcal{C}^2(\RR^2)$ and properties  $ii)$ y $iii)$ of Lemma \eqref{generating} should hold in the whole plane. For this reason we need the following extension lemma. A version of this lemma is stated in \cite[chapter 8]{MatherForni} and for the sake of completeness, we report here a detailed proof (see also \cite{Maro2,Maro3}).

\begin{lemma}\label{extension}
Let $\mathcal{B}^+, \mathcal{B}^- : \RR\longrightarrow \RR$ be $\mathcal{C}^{r}$ diffeomorphisms satisfying \[\mathcal{B}^{\pm}(x+2\pi)= \mathcal{B}^{\pm}(x)+2\pi \] for some $r\geq 2$. Suppose that \[ \mathcal{B}^+(x) > \mathcal{B}^-(x) \qquad \forall x \in \RR.\] 
 Define the following set \[\mathcal{W}=\left\lbrace(x,x_1)\in\RR^2 :  \mathcal{B}^-(x)\leq x_1\leq \mathcal{B}^+(x)\right\rbrace\]
and let $h : \mathcal{W}\longrightarrow \RR$ be a $\mathcal{C}^{r+1}$ function such that:
\begin{itemize}
  
\item 
$h(x+2\pi,x_1+2\pi)=h(x,x_1), \qquad (x,x_1)\in \mathcal{W}$;

\item $ \partial_{x,x_1}h(x,x_1)<0, \qquad (x,x_1)\in\mathcal{W}  $.
\end{itemize}
Then there exists $\tilde{h}\in\mathcal{C}^{r}(\RR^2)$ such that:

\begin{itemize}

\item
$\tilde{h}(x+2\pi,x_1+2\pi)=\tilde{h}(x,x_1), \qquad (x,x_1)\in \RR^2 \ $;

\item $ \partial_{x,x_1}\tilde{h}(x,x_1)<-\delta<0, \quad \mbox{with } \delta>0 \qquad (x,x_1)\in \RR^2 \ $;
\item $\tilde{h}=h$ on $\mathcal{W}$.
\end{itemize}
\end{lemma}
\vspace{0.05cm}
\begin{proof}
  The domain $\mathcal{W}$ is invariant under the translation $T_{2\pi,2\pi}$, in consequence the cocient set $\mathcal{W}/(2\pi\mathbb{Z})^2$ is compact so that
  \[
  \partial_{x,x_1}h(x,x_1)\leq-\delta', \qquad (x,x_1)\in\mathcal{W}
  \]
 for some $\delta'>0$. Consider the $\mathcal{C}^{r-1}$-extension of $\partial_{x,x_1}h(x,x_1)$ to $\mathbb{R}^2$ satisfying the translation invariance under $T_{2\pi,2\pi}$ and keep denoting it $\partial_{x,x_1}h$. By continuity, there exists $\varepsilon>0$, $\delta'\geq\delta>0$ such that $\partial_{x,x_1}h\leq -\delta$ in the domain     
\[\mathcal{W}_{\varepsilon}=\left\lbrace(x,x_1)\in\RR^2 :  \mathcal{B}^-(x)-\varepsilon\leq x_1\leq \mathcal{B}^+(x)+\varepsilon\right\rbrace .\]
\noindent   Consider a $\mathcal{C}^{\infty}$ real valued function $\chi:\RR^2\rightarrow [0,1]$ such that $\chi(x+2\pi,x_1+2\pi)= \chi (x,x_1)$ and
\[
\left\lbrace \begin{array}{l}  \chi = 1  \qquad  (x,x_1)\in\mathcal{W} ,  \\ \chi = 0  \qquad  (x,x_1)\in \RR^2\setminus \mathcal{W}_{\varepsilon} .  \end{array} \right.
\]

\noindent Let's define the function
\[
\mathcal{D}(x,x_1):=\chi\partial_{x\,x_1}h-(1-\chi)\delta \ .
\]
Then by the definition of $\chi$ we have that $\mathcal{D}\in \mathcal{C}^{r-1}(\RR^2)$  and  $\mathcal{D}(x+2\pi,x_1+2\pi)= \mathcal{D} (x,x_1)$. Moreover, 
\[
\left\lbrace \begin{array}{l}  \mathcal{D} = \partial_{x\,x_1} h (x,x_1)  \qquad  (x,x_1)\in\mathcal{W}   \\ \mathcal{D} = -\delta  \qquad \qquad (x,x_1)\in \RR^2\setminus \mathcal{W}_{\varepsilon}.   \end{array} \right.
\]
In particular, with the hypotheses on $h$ we have :
\[
\mathcal{D}\leq -\delta<0 \quad (x,x_1)\in \RR^2\ .
\]
Now, let us consider the Cauchy problems for the wave equation (with periodic boundary conditions):
\begin{equation}\label{waveeq}
\left\lbrace \begin{array}{l} \partial_{x\,x_1} u (x,x_1)= \mathcal{D}(x,x_1),
  \\
  u(x, \mathcal{B}^{\pm}(x))=h(x, \mathcal{B}^{\pm}(x)),
  \\
(\partial_{x_1} u - \frac{1}{(\mathcal{B}^{\pm})'(x)}\partial_{x} u) (x, \mathcal{B}^{\pm}(x))= (\partial_{x_1} h -\frac{1}{(\mathcal{B}^{\pm})'(x)}\partial_{x} h) (x, \mathcal{B}^{\pm} (x)).
\end{array} \right.
\end{equation}

The change of variable
\[
t= \frac{x_1-\mathcal{B}^{\pm}(x)}{2}, \quad y=\frac{x_1+\mathcal{B}^{\pm}(x)}{2}
\]
conjugates system \eqref{waveeq} to the classical wave equation

\begin{equation}\label{waveeq2}
  \left\lbrace \begin{array}{l}
   v_{tt}-v_{yy}= f(t,y),
  \\
  v(0,y)=\phi(y)
  \\
  v_t(0,y)=\psi(y)
\end{array} \right.
\end{equation}
where, denoting  $x(t,y)=(\mathcal{B}^{\pm})^{-1}(y-t)$, $x_1(t,y)=t+y$,
\begin{align*}
  v(t,y) = u\left( x(t,y), x_1(t,y)  \right),  \quad  &f(t,y)= -\frac{4}{(\mathcal{B}^{\pm})'(x(t,y))}  \mathcal{D}\left(x(t,y),x_1(t,y)\right),
  \\
  \phi(y) = h(x(0,y),x_1(0,y)), \quad& \psi(y) = \left(\partial_{x_1} h -\frac{1}{(\mathcal{B}^{\pm})'(x(0,y))}\partial_{x} h\right) (x(0,y), x_1(0,y) ).
\end{align*}
Note that $f,\psi \in \mathcal{C}^{r}$, $\phi\in \mathcal{C}^{r+1}$ and $r\geq 2$ so that problem \eqref{waveeq2} has a unique solution $v(t,y)\in \mathcal{C}^r$ (see \cite{Mikha}). Moreover, since $f(t,y+2\pi)= f(t,y)$, $\phi(y+2\pi)=\phi(y)$ and $\psi(y+2\pi)=\phi(y)$, the solution satisfies $v(t,y+2\pi)=v(t,y)$. Undoing the change of variable, we get a unique solution $u\in \mathcal{C}^{r}(\RR^2)$ of problem \eqref{waveeq} such that $u(x+2\pi,x_1+2\pi)=u(x,x_1)$. Hence, setting $\tilde{h} = u$ proves the lemma. 

\end{proof}

Using the terminology introduced in Theorem \ref{MatherTh}, we recall some of the conclusions of Mather theory

\begin{theorem}[{\bf Mather \cite{Bangert,MatherForni}}]\label{Mather}
  Consider a $\mathcal{C}^2$ function $h:\RR^2\rightarrow \RR$ such that $h(x+2\pi,x_1+2\pi)=h(x,x_1)$ and $\partial^2_{xx_1}h\leq\bar\delta<0$ for all $(x,x_1)\in\RR^2$. Fix $\alpha \in \RR$. Then
     \begin{itemize}
     \item[(i)] if $\alpha=s/q\in\mathbb{Q}$ there exists an increasing sequence $(\overline{x}_n)_{n\in\mathbb{Z}}$ and an homeomorphism of the circle $g_\alpha$ such that
       \begin{itemize}
       \item $g_\alpha(\overline{x}_n) = \overline{x}_{n+1}$ and $\abs{\overline{x}_n - \overline{x}_0-2\pi n\alpha}<2\pi$ for every $n\in\ZZ$,
     \item
      $\partial_{1}h(\overline{x}_n,\overline{x}_{n+1})+ \partial_{2}h(\overline{x}_{n-1},\overline{x}_{n})=0$ and  $\overline{x}_{n+q}= \overline{x}_{n}+2\pi s$ for every  $n\in\mathbb{Z}$;  
      \end{itemize}

    \item[(ii)] If $\alpha \in \RR\setminus \mathbb{Q}$ there exists a set $M_\alpha$ of increasing sequences $x=(\overline{x}_n)_{n\in\mathbb{Z}}$ such that 
      \begin{itemize}
      \item if $x\in M_\alpha$ then $\partial_{1}h(x_n,x_{n+1})+ \partial_{2}h(x_{n-1},x_{n})=0$  for every $n\in\ZZ$, any two translates are comparable and $\abs{x_n - x_0-n2\pi\alpha}<2\pi$ for all $n\in \mathbb{Z}$,
        \item there exists a Lipschitz homeomorphism of the circle $g_\alpha$ with rotation number $\alpha$ and a closed set $A_\alpha\subset\RR$ such that $x\in M_\alpha$ iff $x_0\in A_\alpha$ and $g_\alpha^n(x_0)=x_{n}$ for all $n$, 
     \item the set $Rec(g_\alpha)\subset A_\alpha$ of recurrent points of $g_\alpha$ is either the whole $\RR$ or a Cantor set.
  \end{itemize}

    \end{itemize}
\end{theorem}

\begin{remark}
\em{
    We recall that the homeomorphism $g_\alpha$ satisfies $g_\alpha(x+2\pi)=g_\alpha(x)+2\pi$ for all $x\in\RR$ and the set $Rec(g_\alpha)$ is defined as the set of accumulation points of $\{g_\alpha^n(x)+2\pi k \: :\: (n,k)\in\ZZ^2 \}$ and is independent on the choice of the point $x\in\RR$. Moreover,  
  the condition $\abs{\overline{x}_n - \overline{x}_0-n2\pi\alpha}<2\pi, \; \forall n\in\mathbb{Z}$ implies that
  \[
  \frac{1}{2\pi}\lim_{n\to \infty}\frac{\overline{x}_n}{n}=\alpha,
  \]
  and $\alpha$ is called the rotation number of the orbit.
  }
  \end{remark}

We are now ready for the proof of our result.
\begin{proof}[Proof of Theorem \ref{MatherTh}]

  We apply Lemma \ref{generating} and get the generating function $h$ defined on the set
\[
\mathcal{B} = \left\lbrace(x,x_1)\in \RR^2:\alpha^-(x) < x_1-x <\alpha^+(x) \right\rbrace,
\]
and satisfying the corresponding properties {\it o)--iv)}.\\
Consider the case in which both $W^+,W^-$ are finite. 
Since $W^+ - W^->8\pi$, for every $\alpha$ such that
\[
W^- +4\pi<2\pi\alpha<W^+ - 4\pi,
\]
we can choose $\varepsilon>0$ such that
\[
W^-+\varepsilon < 2\pi\alpha -4\pi<2\pi\alpha+4\pi<W^+ - \varepsilon .
\] 

Now, consider the set
\[
\mathcal{W}=\left\lbrace(x,x_1)\in\RR^2 :  W^- + \varepsilon\leq x_1-x\leq W^+ - \varepsilon)\right\rbrace\subset \mathcal{B},
\]
and apply Lemma \ref{extension} with $\mathcal{B}^\pm(x):=x+W^\pm \mp \varepsilon$ that clearly are diffeomorphisms. We can extend the function $h$ to the whole $\RR^2$ getting a function $\tilde{h}$ satisfying the conditions in Theorem \ref{Mather} and such that $\tilde{h}=h$ in $\mathcal{W}$. 

%

By applying Theorem \ref{Mather} we obtain sequences $(\tilde{x}_n)$,
such that
\[
\partial_{1}\tilde{h}(\tilde{x}_n,\tilde{x}_{n+1})+ \partial_{2}\tilde{h}(\tilde{x}_{n-1},\tilde{x}_{n})= 0, \qquad \forall n\in \mathbb{Z}
\]
and
\[
\abs{\tilde{x}_n - \tilde{x}_0-n2\pi\alpha}<2\pi,\qquad \forall n\in \mathbb{Z}.
\]

From this inequality
 we obtain: 
\[
 2\pi\alpha - 4\pi <\tilde{x}_{n+1} - \tilde{x}_n<2\pi\alpha +4\pi, \qquad \forall n\in \mathbb{Z},
\]
that means that for every $n\in\mathbb{Z}$, $(\tilde{x}_{n+1}, \tilde{x}_n)\in\mathcal{W}$. But since $\tilde{h}=h$ in $\mathcal{W}$,
\[
\partial_{1}{h}(\tilde{x}_n,\tilde{x}_{n+1})+ \partial_{2}{h}(\tilde{x}_{n-1},\tilde{x}_{n})= 0, \qquad \forall n\in \mathbb{Z}.
\]
Hence, in case of rational $\alpha$ we define
\[
\tilde{r}_{n} = f^{-1}(- \partial_{1}h(\tilde{x}_n,\tilde{x}_{n+1}))=f^{-1}(- \partial_{1}h(\tilde{x}_n,g_\alpha(\tilde{x}_{n})))
\]
such that $(\tilde{x}_n,\tilde{r}_n)\subset\Sigma$ is the $(s,q)$-periodic orbit of $\Phi$.
In the irrational case, the set
$\tilde{\mathcal{M}}_\alpha$ is given by
\[
\tilde{\mathcal{M}}_\alpha =\{(\xi,\eta)\in\mathbb{R}^2 \: : \: \xi\in Rec(g_\alpha),\quad \eta = f^{-1}(- \partial_{1}h(\xi,g_\alpha(\xi)))  \}\subset\Sigma.
\]
Note that the Lipschitz regularity of $f^{-1}$ plays a role at this stage.

  In case $-\infty<W^-<W^+=\infty$ is enough to choose $\alpha$ such that $2\pi\alpha>W^-+ 4\pi$ and fix $M>2\pi\alpha + 8\pi$. Fix $\varepsilon$ such that
\[W^-+\varepsilon\leq 2\pi\alpha-4\pi \]
and apply extension lemma with
$\mathcal{B}^-(x) = x+W^- + \varepsilon$ and $\mathcal{B}^+(x) =x+ M $, in order to get the same result.

The other cases are similar.

\end{proof}


\section{Exact symplectic properties of the Poincar\'e map}
\label{sec:exact}

Fix $r\geq 2$ and a $\mathcal{C}^{r+1}$ function 
$ f:\, (a,b)  \longrightarrow  \RR $ such that $f'(r)$ never vanishes and consider the associated differential form
$\tilde{\lambda} = \mathrm{d}f(r) \wedge \mathrm{d}\theta = f'(r) \mathrm{d}r \wedge \mathrm{d}\theta$ on $(a,b)\times\mathbb{R}$.
In local coordinates, the corresponding time dependent Hamiltonian system takes the form  

\begin{equation} \label{PVI}
\left\lbrace \begin{array}{l}  \dot{r}= \frac{1}{f'(r)} \,  \partial_{\theta} H (t,r,\theta)  \\  \dot{\theta}=- \frac{1}{f'(r)} \, \partial_r H (t,r,\theta), \\ r(0)=r_0 \\\theta(0)=\theta_0.   \end{array} \right.
\end{equation}

\noindent Suppose that $H:\mathbb{R}\times (a,b)\times\mathbb{R}\rightarrow \mathbb{R}$ is continuous in $t$ and $\mathcal{C}^{r+1}$ in the phase variables $(r,\theta)$ and the following periodicity hold
\[
H (t+1,r,\theta)=H (t,r,\theta)\quad\mbox{and}\quad H (t,r,\theta+2\pi)=H (t,r,\theta).
\]
By the periodicity in $\theta$ we have the phase space is the cylinder $(a,b)\times\mathbb{T}$.  

\begin{remark}
\em{
  In our problem $(a,b)=(r_\ast,\infty)$, $f(r)=-\frac{1}{4r}$ and
  \[
  H (t,r,\theta)= - \frac{1}{2}\ln(2r) + \: p\left(t, \frac{\cos\theta}{\sqrt{2r}},-\frac{\sin\theta}{\sqrt{2r}} \right).
  \]
  }
\end{remark}

Let us consider the Poincar\'e map $\mathcal{P}(r_0,\theta_0) = (r_1,\theta_1)$ associated to  the Cauchy problem \eqref{PVI}. 
%
 By the hypothesis on $H$ and $f$, the map $\mathcal{P}$ belongs to $\mathcal{C}^r((a,b)\times\mathbb{R})$ and satisfies:
  \[
  \mathcal{P}(r_0,\theta_0+2\pi) = \mathcal{P}(r_0,\theta_0)+(0,2\pi).
  \]

\begin{lemma}\label{exactsymp}
  The Poincar\'e map $\mathcal{P}$ is exact symplectic with respect to the form $\tilde{\lambda}$.
\end{lemma}

\begin{proof}

    To simplify the notation, let us denote $(r(t),\theta(t))$ the solution \\$(r(t;r_0,\theta_0), \theta(t;r_0,\theta_0))$  of (\ref{PVI}).
    Consider the $\mathcal{C}^r$ function
    \[\mathcal{S}(r_0,\theta_0) = -\int_0^1 \left[\frac{f(r(t))}{f'(r(t))} \partial_r H(t, r(t), \theta(t) ) - H(t, r(t), \theta(t)) \right]\mathrm{d}t, \]
    and note that by the periodicity assumptions in $\theta$, and the uniqueness, we have 
    \[\mathcal{S}(r_0,\theta_0+2\pi) = \mathcal{S}(r_0,\theta_0).
    \]
	Let us now prove that  
    \[
    \mathrm{d}\mathcal{S}(r_0,\theta_0)= f(r_1)\,\mathrm{d}\theta_1 - f(r_0)\,\mathrm{d}\theta_0.
    \]
    We start with
    \begin{align} \label{desexact}
          \nonumber \partial_{r_0} & \mathcal{S}(r_0,\theta_0) = -\int_0^1 \left[ - \frac{f(r(t))f''(r(t))}{[f'(r(t))]^2} [\partial_{r_0} r(t)] \partial_r H(t, r(t), \theta(t))  \right.\\  & \left. +\frac{f(r(t))}{f'(r(t))} \partial_{r_0}\left[\partial_r H(t, r(t), \theta(t))\right] -  \partial_{\theta} H(t, r(t), \theta(t))\left[\partial_{r_0}\theta(t)\right]\right]\,\mathrm{d}t.
    \end{align}
Note that the term $\partial_{r} H(t, r(t), \theta(t))[\partial_{r_0}r(t)] $ is canceled. Now, using the first equation in \eqref{PVI}, and integrating by parts the last term, we obtain
  \begin{align*} 
   \int_0^1 \partial_{\theta} H(t, r(t), \theta(t))&\left[\partial_{r_0}\theta(t)\right]\,\mathrm{d}t =  \int_0^1  \dot{f}(r(t)) \left[\partial_{r_0}\theta(t)\right]\, \mathrm{d}t \\ &= \left[  f(r(t)) \partial_{r_0}\theta(t) \right]_{t=0}^{t=1} - \int_0^1  f(r(t)) \left[\partial_{r_0}\dot{\theta}(t)\right] \, \mathrm{d}t.
  \end{align*}
  Replacing in  \eqref{desexact} and using the second equation in \eqref{PVI}, we get
     \[
          \partial_{r_0} \mathcal{S}(r_0,\theta_0) = \left[ f(r(t)) \partial_{r_0}\theta(t) \right]_{t=0}^{t=1}.
   \]   
    Analogously,
      \[
      \partial_{\theta_0} \mathcal{S}(r_0,\theta_0) = \left[ f(r(t))  \partial_{\theta_0}\theta(t) \right]_{t=0}^{t=1}.
      \]
      Hence,
      \begin{align*}
        \mathrm{d}\mathcal{S}(r_0,\theta_0) &=
      \partial_{r_0} \mathcal{S}\,\mathrm{d}r_0 + \partial_{\theta_0} \mathcal{S}\,\mathrm{d}\theta_0
        = \left[ f(r_1)
\partial_{r_0}\theta_1  - f(r_0) \partial_{r_0}\theta_0  \right]\mathrm{d}r_0 + \\ & +
       \left[ f(r_1) \partial_{\theta_0}\theta_1  -  f(r_0) \partial_{\theta_0}\theta_0 \right] \mathrm{d}\theta_0 = f(r_1)\,\mathrm{d}\theta_1 - f(r_0)\,\mathrm{d}\theta_0.
      \end{align*}
\end{proof}

\section{The twist property for the vortex problem}
\label{sec:twist}



In this section we consider the Poincar\'e map $\cal{P}$ associated to system \eqref{cuatro}. We recall the notation
    
\[ \begin{array}{rcl} \mathcal{P}: \,\,\Sigma(a_\ast)=]a_\ast,+\infty[\times\mathbb{T} & \longrightarrow & \RR^2 \\ (r_0,\theta_0) & \longmapsto & (r_1,\theta_1)= \left(\mathcal{F}(r_0,\theta_0),\mathcal{G}(r_0,\theta_0)\right). \end{array}\]

The following theorem will clearly imply the twist condition \eqref{deftwist},

    \begin{theorem}\label{twisttheo}
     Suppose that $p\in \mathcal{R}^{2}_{\varepsilon}$ and the origin is a zero of order $4$.
Then
    \begin{equation}\label{twist}
      \frac{\partial\mathcal{G}}{\partial r_0}\underset{r_0\to\infty}{\longrightarrow} 2\, , \,\, \mbox{uniformly in }  \theta_0.
    \end{equation}
    
\end{theorem}


    To prove the theorem, let us fix a solution $(r(t;\theta_0,r_0),\theta(t;\theta_0,r_0))$ of problem \eqref{cuatro}. Note that, since $p\in \mathcal{R}^{2}_{\varepsilon}$, the vector field of system \eqref{cuatro} is $\mathcal{C}^1$ in the variables $(r,\theta)$ so that the solution is unique and 
    we can consider the associated variational equation
    \begin{equation} \label{vareq}
\left\lbrace \begin{array}{l} \dot{\mathit{Y}}=
\mathcal{M} (t,r(t;r_0,\theta_0),\theta(t;r_0,\theta_0)) \mathit{Y},  \\  \mathit{Y}(0)=\mathbb{I}_2  \end{array} \right.
    \end{equation}
    where
    \[
    \mathcal{M}(t;r,\theta) = \frac{\partial\, (F(t,r,\theta),2r + G(t,r,\theta)) }{\partial\,(r,\theta)}
    \]
is the Jacobian of the vector field in \eqref{cuatro}.
We denote the matrix solution \[ \mathcal{Y}(t;r_0,\theta_0) = \begin{pmatrix} \partial_{r_0}r(t;r_0,\theta_0) & \partial_{\theta_0}r(t;r_0,\theta_0) \\\partial_{r_0}\theta(t;r_0,\theta_0) & \partial_{\theta_0}\theta(t;r_0,\theta_0) \end{pmatrix}\]
and by the definition of the Poincar\'e map,
\[
\frac{\partial\mathcal{G}}{\partial r_0}(r_0,\theta_0) = \partial_{r_0}\theta(1;r_0,\theta_0).
\]
In the integrable case $p=0$, the Jacobian matrix is \\$A=\begin{pmatrix} 0 & 0 \\  2& 0     \end{pmatrix}$ 
and the solution of the corresponding variational equation is
\begin{equation}\label{solinit}
\mathcal{Y}_{int}(t;r_0,\theta_0)= \begin{pmatrix} 1 & 0 \\2t& 1 \end{pmatrix},
\end{equation}
that shows that the Poincar\'e map of the unperturbed problem is twist.

To prove the result in the non integrable case, we will follow a perturbative approach. More precisely, we will prove that the solution remains close to that of the integrable case over a period $t\in[0,1]$. For this purpose we begin considering the following splitting. To simplify the notation we will denote a solution of \eqref{cuatro} by $(r(t),\theta(t))$ where we have dropped the dependence on the initial conditions.

\begin{lemma}\label{lemsplit}
Under the hypothesis of Theorem \eqref{twisttheo}, we have the following splitting:
\[\mathcal{M}(t;r(t),\theta(t)) = A + B(t,r_0, \theta_0) + C(t,r_0,\theta_0)\]
where $B(t,r_0,\theta_0) $ is bounded, the entries satisfy $b_{11}=b_{21}=b_{22}=0$ and $\forall\varphi\in \mathcal{C}^{\infty}([0,1])$ 
 \begin{equation}\label{propint}
   \abs{ {\int_0^{t} b_{12}(s,r_0,\theta_0)\,\varphi(s) \, \mathrm{d}s}}
  \underset{r_0\to\infty}{\longrightarrow} 0 \quad \mbox{uniformly in } t\in [0,1], \theta_0\in\TT.
 \end{equation}

Moreover, 

\begin{equation*}
\Vert C (t,r_0,\theta_0)\Vert  \underset{r_0\to\infty}{\longrightarrow} 0 \quad \mbox{uniformly in } t\in [0,1], \theta_0\in\TT.
\end{equation*}
\end{lemma}
\begin{proof}

Since the origin is a zero of order $4$ for $p$, we can split the perturbation as
\[ p(t,x,y) = T_4(t,x,y) + \tilde{p}(t,x,y),\]
where $T_4$ is a homogeneous polynomial of degree $4$. From system \eqref{cuatro} 
\[
F(t,r,\theta) = 4r^2 \partial_{\theta}p\left[\left(t, \frac{\cos \theta}{\sqrt{2r}}, \frac{-\sin \theta}{\sqrt{2r}}\right)\right],
\]
so that $p$ induce the following splitting on $F$
\begin{equation*}
F(t,r,\theta) = F_{\ast}(t,\theta) + \tilde{F}(t,r,\theta),
\end{equation*}
where, using the homogeneity of $T_4$ w.r.t the variable $r$,
\[
F_{\ast} (t,\theta) = \partial_{\theta}[ T_4\left(t, \cos \theta , -\sin \theta\right)], \quad
\tilde{F}(t, r, \theta) = 4r^2 \partial_{\theta}\left[\tilde{p}\left(t, \frac{\cos \theta}{\sqrt{2r}}, \frac{-\sin \theta}{\sqrt{2r}}\right)\right].
\]

Therefore we write
\begin{align*}
\mathcal{M}(t,r(t),\theta(t)) &=  \left.\begin{pmatrix} \partial_{r}F(t,r,\theta) & \partial_{\theta}F(t,r,\theta) \\ 2 + \partial_{r}G(t,r,\theta)& \partial_{\theta}G(t,r,\theta)\end{pmatrix} \right\vert_{(r(t),\theta(t))}
\\ & = \begin{pmatrix} 0 & 0 \\2& 0 \end{pmatrix} +\left[\begin{pmatrix}0 & b_{12} \\ 0 & 0 \end{pmatrix} +  \begin{pmatrix} c_{11} & c_{12} \\c_{21}& c_{22} \end{pmatrix}\right]_{(r(t),\theta(t))}
\end{align*}
where $b_{12}=b_{12}(t,r,\theta)$ and $c_{ij}=c_{ij}(t,r,\theta)$ defined as 
\[
b_{12}   := \partial_{\theta} F_{\ast}(t,\theta)=\partial_{\theta \theta} \left[T_4\left(t, \cos \theta , -\sin \theta \right)\right]
\]
and
\begin{align*}
  c_{11}  & := \partial_{r}\tilde{F}(t,r,\theta) 
  \qquad&   c_{12}   := \partial_{\theta} \tilde{F}(t,r,\theta) 
  \\
   c_{21}  & :=  \partial_{r}G(t,r,\theta) 
 \qquad&  c_{22}   := \partial_{\theta}G(t,r,\theta)
\end{align*} 

\noindent Concerning the matrix $C$, one can first explicitly write the expression of the entries $c_{ij}$ recalling that $\tilde{F}$ has just been introduced and $G$ is defined in \eqref{defg}. Note that they depend on the derivatives up to second order of the functions $p\left(t, \frac{\cos \theta}{\sqrt{2r}}, \frac{-\sin \theta}{\sqrt{2r}}\right)$, $\tilde{p}\left(t, \frac{\cos \theta}{\sqrt{2r}}, \frac{-\sin \theta}{\sqrt{2r}}\right)$   w.r.t the variables $(r,\theta)$. These derivatives are estimated in Lemma \ref{lemder} of Appendix A and can be used to get the following estimate
\[
r^{3/2} \abs{c_{11}} + r^{1/2}\abs{c_{12}} + r^2 \abs{c_{21}} + 
r \abs{c_{22}} 
\leq K
\]
with $K$ independent on $r,\theta$. We evaluate the entries $c_{ij}$ on a solution $(r(t),\theta(t))$ and we remember that from Lemma \ref{evosol} we have that for $r_0>a_\ast$
\begin{equation}
\abs{r(t)- r_0} \leq K_1, \:\quad \forall \: \theta_0 \in \TT \: \textnormal{and} \: t\in [0,1].
\end{equation}
This proves that
$ \Vert C (t,r_0,\theta_0)\Vert \longrightarrow 0$ as $r_0\to\infty$ uniformly in $\theta_0\in \TT, t\in[0,1]$.

Let us study the matrix $B$. Since $T_4\left(t, \cos \theta , -\sin \theta\right)$ is a trigonometric polynomial, $\abs{b_{12}}\leq K$ and $B$ is bounded.
To obtain \eqref{propint} we note that since $T_4\left(t, \cos \theta , -\sin \theta\right)$ is a trigonometric polynomial of degree $4$, $\partial_{\theta} \left[T_4\left(t, \cos \theta , -\sin \theta\right)\right]$ will be another trigonometric polynomial (of the same degree) that we denote $\tilde{T}_4\left(t, \cos \theta , -\sin \theta\right)$. 
Let us define
\[
P_4 (t, \eta, \xi) := \tilde{T}_4(t, \eta , \xi)
\]
We show that we can apply Lemma \ref{RL} (see Appendix A) choosing the polynomial of degree $N = 4$
\[q(t,\eta,\xi) = -\xi\frac{\partial P_4}{\partial \eta}(t,\eta,-\xi) - \eta\frac{\partial P_4}{\partial \xi}(t,\eta,-\xi).
\]
Since  $P_4 (t, \cos \theta, -\sin \theta)$ is a periodic primitive in the variable $\theta$ of the function  $q (t, \cos \theta, \sin \theta)$ the condition \eqref{izero} holds. Moreover, the regularity assumption in Definition \ref{deforder} guarantee that $q$ has $\mathcal{C}^1$  coefficients in the variable $t$.

\noindent Let us define the function $\beta (t) := \theta (t) -  2r_0t$. 
The estimate \eqref{lemma1} on the angular evolution 
gives a bound of $\norm{\beta}_{\infty}$. To get a bound of  $\dot{\norm{\beta}}_{\infty}$ we observe that
\[
\dot{\beta}(t) = \dot{\theta}(t) - 2r_0 = 2\left( r(t) - r_0 \right) + G(t,r(t),\theta(t)).
\]
And again from \eqref{lemma1} and \eqref{bound33}, we have 
\[
\dot{\norm{\beta}}_{\infty} \leq 2K + \frac{C_1}{(r_0-K)}.
\]
Finally, Lemma \ref{RL} in Appendix A can be applied to deduce that for all $\varphi \in \mathcal{C}^{\infty}([0,1])$ and $t\in [0,1]$, 
\begin{align*}
  \abs{\int_0^tb_{12}(s,r_0,\theta_0)\varphi(s)\mathrm{d}s }&=\abs{\int_0^tb_{12}(s,\theta(s))\varphi(s)\mathrm{d}s } \\
  &=\abs{\int_0^t\partial_{\theta \theta} T_4\left(s, \cos \theta(s) , -\sin \theta(s) \right)\varphi(s)\mathrm{d}s }
  \\
  & = \abs{\int_0^t \partial_\theta P_4 (s, \cos\theta(s),-\sin\theta(s))\varphi(s) \mathrm{d}s} \\
  &=\abs{\int_0^t q(s, \cos\theta(s),\sin\theta(s))\varphi(s) \mathrm{d}s} \\
  &=\abs{\int_0^t q(s, \cos(2r_0s+\beta(s)),\sin(2r_0s+\beta(s)))\varphi(s) \mathrm{d}s}\\
  &\leq \frac{C_{RL} }{2r_0} .
\end{align*}
\end{proof}

Let us write the variational equation \eqref{vareq} as

\begin{equation*}
\left\lbrace \begin{array}{l} \dot{\mathit{Y}}=
\left(  A + B(t,r_0, \theta_0) + C(t,r_0,\theta_0) \right) \mathit{Y},  \\  \mathit{Y}(0)=\mathbb{I}_2.  \end{array} \right.
\end{equation*}
We claim that the solution $\mathcal{Y} (t;r_0,\theta_0)$ converge uniformly,  as $r_0 \to \infty$, to the solution $\mathcal{Y}_{int}(t;r_0,\theta_0)$ of the integrable case \eqref{solinit}.

We prove it applying Lemma \ref{lemconv} of Appendix to the family of matrices $M(t,r_0,\theta_0)=B(t,r_0, \theta_0) + C(t,r_0,\theta_0)$ as $r_0\to +\infty$. From Lemma \ref{lemsplit} we have that $M(t,r_0,\theta_0)$ is uniformly bounded and $C(t,r_0,\theta_0)$ converge uniformly to $0$ so that it converge also in the weak* topology. To study the convergence of the matrix $B(t,r_0, \theta_0)$ it is enough to consider the term $b_{12}$.  
From \eqref{propint} and using the density of $\mathcal{C}^\infty$ in $\mathcal{L}^1$ we have
 \begin{equation*}
\abs{\int_0^{t} b_{12}(s,r_0,\theta_0)\,\varphi(s) \, \mathrm{d}s}  \underset{r_0\to\infty}{\longrightarrow} 0,  \qquad \forall\varphi\in \mathcal{L}^{1}([0,1]) ,\quad t\in [0,1].
 \end{equation*}
Hence, we can apply Lemma \ref{lemconv} and get 
\[ \mathcal{Y} (t;r_0, \theta_0)
\underset{r_0\to\infty}{\longrightarrow}
\mathcal{Y}_{int}(t;r_0,\theta_0) \quad \mbox{uniformly in } t\in [0,1], \theta_0\in\TT,\]
from which \eqref{twist} follows evaluating in $t=1$.

\section{Proof of the main Theorem }
\label{sec:proof}

In this section we apply Theorem  \ref{MatherTh} and Corollary \ref{MatherCor} to the Poincar\'e map $\cal{P}$ of system \eqref{cuatro} and get the so called Aubry-Mather orbits of rotation number $\alpha$. These orbits determine the solutions we announced in our main theorem \ref{Theo}.

By Lemma \ref{evosol} the map $\cal{P}$ is well defined in $\Sigma(a_*)$ and is a $\mathcal{C}^2$-diffeomorphism since $p\in\mathcal{R}^3_\varepsilon$. For every initial condition in $\Sigma(a_*)$, the corresponding solution satisfies $|r(t)-r_0|\leq K$ for all $t\in[0,1]$. Moreover, from \eqref{bound33}, we can find $a_1>0$ such that, in $\Sigma(a_1)$,
\[
|\dot\theta| = |2r+G(t;r,\theta)| \geq 2r-\frac{C_1}{2r} > 0.
\]
By theorem \ref{twisttheo}, the map $\cal{P}$ is twist in $\Sigma(a_2)$ for some $a_2$ large enough.
Let us consider the strip $\Sigma(\bar{r})$ where $\bar{r}=\max\{a_*,a_1,a_2\}+K$.
Theorem \ref{twisttheo} also imply that the following limits hold: 
  \begin{align*}
W^+ :=\min_x \{ \lim_{r \to +\infty} \mathcal{G}(r,x)- x \}= +\infty, \\ 
W^-:=\max_x \{\lim_{r \to \bar{r}} \mathcal{G}(r,x)- x\} = c< +\infty.
\end{align*}
  so that $W^+-W^->8\pi$.

Since, from Lemma \ref{exactsymp}, the map $\cal{P}$ is exact symplectic w.r.t the form \\$\lambda =\frac{1}{4r^2}dr\wedge d\theta =d(-\frac{1}{4r})\wedge d\theta $  and $1/(4r)$ is Lipschitz for $r>\bar{r}$,  we can apply Theorem \ref{MatherTh} and Corollary \ref{MatherCor} to the Poincar\'e map restricted to the strip $\Sigma_{\bar{r}}$.
For every   
$\alpha>(c+2)/2\pi$, we get two functions $\phi,\eta:\RR\rightarrow\RR$ such that, for every $\xi\in\mathbb{R}$
  \begin{align}
    &\label{thcor1} \phi(\xi+2\pi) = \phi(\xi)+2\pi, \quad \eta(\xi+2\pi)=\eta(\xi) , \\
    &\label{thcor2} \mathcal{P} (\phi(\xi),\eta(\xi))=(\phi(\xi+2\pi\alpha),\eta(\xi+2\pi\alpha)).
  \end{align}
  
  For every $\xi\in\RR$, let us consider the solution of the Cauchy problem \eqref{cuatro} with initial condition $(r(0),\theta(0))=(\eta(\xi),\phi(\xi))$ and denote it $(r(t),\theta(t))_\xi$. By \eqref{thcor1} and uniqueness we have that
  \[
  (r(t),\theta(t))_{\xi+2\pi} =(r(t),\theta(t))_\xi+(0,2\pi), 
  \]
  and from \eqref{thcor2} and the definition of $\cal{P}$,
  \begin{equation*}
(r(t+1),\theta(t+1))_{\xi}=(r(t),\theta(t))_{\xi+2\pi\alpha}
  \end{equation*}
  so that conditions \eqref{contheo}-\eqref{contheo2} are satisfied.

  The function $\xi\mapsto\Phi_\xi(a,b)$ introduced in Remark \ref{rem2} has the same regularity of the functions $\phi,\eta$ that can have at most jump discontinuities. Moreover, from properties \eqref{thcor1},\eqref{thcor2}, if $\xi$ is a point of continuity, so are $\xi + 2\pi$ and $\xi+2\pi\alpha$.

Finally, these solutions have rotation number $\alpha$, actually,
\begin{align*}
  \lim_{t\to\infty}\frac{\theta_\xi(t)}{t} &= \lim_{k\to\infty}\frac{\theta_\xi(k)}{k} =  \lim_{k\to\infty}\frac{\theta_{\xi+2\pi k\alpha}(0)}{k} =   \lim_{k\to\infty}\frac{\phi(\xi+2\pi k\alpha)}{k} \\
  &= \lim_{k\to\infty}\frac{\phi(\xi+2\pi \{k\alpha\})+ 2\pi [k\alpha] }{k} =     2\pi\alpha.
  \end{align*}
  where $[x]$ denote the integer part of $x$ and $\{x\}=x-[x]$.

\section{Conclusions} \label{sec:conclu}

This paper can be seen as an example of the study of twist dynamics around a singularity in Hamiltonian systems. As a paradigmatic example we choose the point-vortex model. In suitable variables we applied a version of Aubry-Mather theory to get similar results as in the case of exact area-preserving maps of the annulus \cite{Mathertop}. 

We suppose that the origin was a zero of order $4$ for the perturbation. This condition played a role in the regularizing change of variable $\varphi$. For this reason, it seems unclear how to weaken this assumption. 

Our result leaves open the distinction between classical and generalized quasi-periodic solutions. This relies on the nature of the corresponding Mather set with irrational rotation number. Actually it can be either a invariant curve or a Cantor set. A possible future development of the present work  could be finding conditions that break invariant curves.

\section*{Acknowledgment}
We are grateful to Rafael Ortega for fruitful discussions and suggestions.

\appendix
\renewcommand{\thesection}{\Alph{section}}
\section{Appendix}
This appendix is intended to present some technical results that are required for the proof of the twist condition in Section \ref{sec:twist}. 

\begin{lemma} \label{lemder}
  Suppose the $p\in \mathcal{R}^{2}_{\varepsilon}$ and that the origin is a zero of order $N$. Consider the decomposition given in Definition \ref{deforder}
  \[
  p(t,x,y) = T_N(t,x,y) +\tilde{p}(t,x,y)
  \]
  and set the functions
  \[ x = x(r,\theta)= \frac{\cos \theta}{\sqrt{2r}},\quad y =y(r,\theta)= \frac{-\sin \theta}{\sqrt{2r}}
  \]
  defined for $  r > \frac{1}{2\varepsilon^2}$ and $\theta\in \TT$. Then, there exists a constant $C>0$ such that 
\begin{itemize}

  \item [\textbf{1)}] $r^{(N+1)/2}\left(\abs{\partial_{\theta}\, \tilde{p}(t,x,y)} + \abs{\partial_{\theta \theta}\, \tilde{p}(t,x,y)}\right) \leq  C$,
  \item [\textbf{2)}] $r^{(N+3)/2}\abs{\partial_{r\theta}\, \tilde{p}(t,x,y)}\leq C$,
  \item [\textbf{3)}] $r^{(N+2)/2}\left(\abs{\partial_{r}\, p(t,x,y)} + \abs{\partial_{r\theta}\, p(t,x,y)}\right)\leq C$,
   
  \item [\textbf{4)}] $r^{(N+4)/2}\abs{\partial_{rr}\, p(t,x,y)}\leq C$. \end{itemize}

\end{lemma}

\begin{proof}

To get the estimate \textbf{1)}, let's compute explicitly the derivatives with respect to $\theta$:
\[ \partial_{\theta} \tilde{p}(t,x(r,\theta),y(r,\theta)) = -\frac{1}{(2r)^{1/2}}\left[\sin\theta \, \partial_{x} \tilde{p}(t,x,y) + \cos\theta \, \partial_{y} \tilde{p}(t,x,y) \right] \]
and 
\begin{align*}
\partial_{\theta\theta} & \tilde{p}(t,x(r,\theta),y(r,\theta)) =  \frac{1}{(2r)^{1/2}}\left[\sin\theta \, \partial_{y} \tilde{p}(t,x,y) - \cos\theta \, \partial_{x} \tilde{p}(t,x,y) \right]
\\ & + \frac{1}{(2r)}\left[\sin^2\theta \, \partial_{xx} \tilde{p}(t,x,y) + 2\cos\theta  \sin\theta \,\partial_{xy} \tilde{p}(t,x,y) + \cos^2\theta \,\partial_{yy} \tilde{p}(t,x,y) \right].
\end{align*} 
From the definition of zero of order $N$ we have: 
\begin{align*}
  \abs{\partial_{\theta}\, \tilde{p}(t,x,y)} & + \abs{\partial_{\theta \theta}\, \tilde{p}(t,x,y)}  \leq  \frac{C_1}{r^{1/2}} \left(\abs{\partial_x \, \tilde{p}(t,x,y)} + \abs{\partial_y \, \tilde{p}(t,x,y)}\right) \\& + \frac{C_2}{r} \left(\abs{\partial_{xx} \, \tilde{p}(t,x,y)}  + \abs{\partial_{xy} \, \tilde{p}(t,x,y)} + \abs{\partial_{yy} \, \tilde{p}(t,x,y)}\right) \\& \leq \frac{C_1}{r^{1/2}}(|x|^N+|y|^N)+\frac{C_1}{r}(|x|^{N-1}+|y|^{N-1})
  \\& \leq \frac{C_1}{r^{(N+1)/2}} + \frac{C_2}{r^{1+(N-1)/2}} \leq  \frac{C}{r^{(N+1)/2}}
\end{align*}

To obtain \textbf{2)}, \textbf{3)} and \textbf{4)}, the computations are similar. 
%
%
%
%
%
\end{proof}

\vspace{0.5cm}
The following result is a lemma of Riemann-Lebesgue type.

\begin{lemma} \label{RL}
Let $q(t,\eta,\xi)$ be a polynomial of degree $N$,
\[q(t,\eta,\xi) = \sum_{j+h \leq N} \alpha_{j,h}(t) \eta^j \xi^h\]
 with $\alpha_{j,h}(t)\in \mathcal{C}^1 (\RR/\mathbb{Z})$.
 Assume in addition that for each $t$
 \begin{equation}{\label{izero}}
 \int_0^{2\pi} q (t, \cos \, \theta, \sin \, \theta)\, \mathrm{d}\theta = 0 \, .
 \end{equation}
 Let $\beta\in \mathcal{C}^1 ([0,\tau])$ with $[0,\tau] \subset [0,1]$ and $\varphi\in \mathcal{C}^{\infty}([0,1])$. Then there exists $C_{RL}>0$ such that
 \[\abs{\int_0^{t} q (s, \cos \,(\lambda s + \beta(s)), \sin \, (\lambda s + \beta(s)))\,\varphi(s) \, \mathrm{d}s} \leq \frac{C_{RL} }{\abs{\lambda}}\]
 if $t\in [0,\tau]$ and $\lambda \in \mathbb{R}\setminus \left\lbrace 0 \right\rbrace$. Moreover, the constant $C_{RL}$ depends upon $N$, $\max_{j,h} \left[\norm{\alpha_{j,h}}_{\infty} +  \norm{\dot{\alpha}_{j,h}}_{\infty} \right]$, $\norm{\beta}_{\infty}$, $\Vert \dot{\beta} \Vert_{\infty}$ $\norm{\varphi}_{\infty}$ and $\Vert \dot{\varphi} \Vert_{\infty}$.
\end{lemma}

\vspace{0.3cm}

\begin{proof}
The function $q (t, \cos \, \theta, \sin \, \theta)$ has a finite Fourier expansion with \\respect to $\theta$, say
\[q (t, \cos \, \theta, \sin \, \theta) = \sum_{\abs{k}\leq N} q_k (t) e^{ik\theta}. \]
The coefficients $q_k$ can be expressed in terms of the functions $\alpha_{j,h}$ and belong to $\mathcal{C}^1 (\RR/\mathbb{Z})$,
\[q_k (t) = \frac{1}{2\pi}  \int_0^{2\pi} q (t, \cos \, \theta, \sin \, \theta) e^{-ik\theta} \, \mathrm{d}\theta.\]
The condition \eqref{izero} implies that $q_0 (t)$ vanishes everywhere and so the integral $I(t)$ we want to estimate can be expressed as the sum
\[I(t) = \sum_{0 < \abs{k} \leq N} I_k (t) \:\quad \textnormal{with} \:\quad I_k (t) = \int_0^{t} q_k (s) e^{ik\beta(s)} e^{ik\lambda s} \,\varphi(s) \, \mathrm{d}s. \]
Since we have excluded $k = 0$ these integrals can be estimated by a standard procedure in the theory of oscillatory integrals, see for instance \cite{Arnsing}. After integrating by parts
\begin{align*}
I_k (t) &= \frac{1}{ik\lambda} \left[ q_k (t) e^{ik\beta (t)} e^{ik\lambda t}\,\varphi(t) -  q_k (0) e^{ik\beta (0)}\,\varphi(0) \right.
\\ &  \left.
-  \int_0^{t} \left(q_k (s) e^{ik\beta(s)}\,\varphi(s)\right)^{\prime} e^{ik\lambda s} \, \mathrm{d}s  \right].
\end{align*}
Therefore, 
\[\abs{I_k (t)} \leq \frac{C_k}{\abs{k} \abs{\lambda}}\]
with 
\[C_k = \norm{q_k}_{\infty} \left[2\norm{\varphi}_{\infty}  + \abs{k}\norm{\dot{\beta}}_{\infty}\,\norm{\varphi}_{\infty} + \norm{\dot{\varphi}}_{\infty}\right] + \norm{\dot{q}_k}_{\infty}\,\norm{\varphi}_{\infty} . \]

\end{proof}


Finally, we state the following lemma  concerning the uniform convergence of the solution of a linear ODE whose time-dependent coefficients are bounded and converging weak* in $\mathcal{L}^\infty$. Similar results can be found in \cite{Zhang} and \cite{Ortega92}. For the proof we will follow the lines of the proof of Lemma 2.1 in \cite{Ortega92}.

Consider the following linear system depending on the parameters \\$(r,\theta)\in (a,b)\times\TT$, $a>b$
\begin{equation} \label{sysweak}
\left\lbrace \begin{array}{l} \dot{\mathit{Y}}= \left( A + M (t;r,\theta)\right)\mathit{Y},  \\  \mathit{Y}(0)= \mathbb{I}_2 , \end{array} \right.
\end{equation}

where $A$, $M$ are $2\times 2$ matrices, $A$ is constant and $M\in \mathcal{C}^1([0,1]\times (a,b)\times\TT)$. We denote the matrix solution of this system as $\mathcal{Y} (t;r,\theta)$. 

\begin{lemma}\label{lemconv}

Suppose that the family $\left\lbrace M(t;r, \theta)\right\rbrace$ is uniformly bounded in $\mathcal{L}^\infty([0,1])$ and that $M(t;r,\theta)$ converges to $\tilde{M}(t;\theta)\in \mathcal{L}^\infty([0,1])$ in the weak* sense as $r\to b$. Then
\[ \mathcal{Y} (t;r,\theta)
\underset{r\to b}{\longrightarrow}
\tilde{\mathcal{Y}}(t;\theta) \quad  \text{uniformly,} \qquad t\in [0,1] \quad \forall \theta\in\TT\]
where $\tilde{\mathcal{Y}} (t;\eta)$ is the matrix solution of the problem 
\begin{equation} \label{sysweakn}
\left\lbrace \begin{array}{l} \dot{\mathit{Y}}= \left(A + \tilde{M} (t; \theta)\right)\mathit{Y} , \\  \mathit{Y}(0)=\mathbb{I}_2\ . \end{array} \right.
\end{equation}

\end{lemma}
\begin{proof}
 The solution of \eqref{sysweak} can be written as: 
\begin{equation}\label{fdr}
\mathcal{Y} (t;r,\theta)= \mathbb{I}_2  + \int_0^t  (A + M (s; r,\theta))\mathcal{Y} (s;r,\theta)  \, \mathrm{d}s, \quad t\in\RR,
\end{equation}
from which we get the following estimate on the matrix norm
\begin{equation}
\Vert \mathcal{Y} (t;r,\theta)\Vert \leq 1  + \int_0^t  \Vert A + M (s;r, \theta) \Vert \, \Vert\mathcal{Y} (s;r,\theta)\Vert  \, \mathrm{d}s.
\end{equation}
Gronwall lemma applied on the interval [0,1] gives us
\[ 
\Vert \mathcal{Y} (t;r,\theta) \Vert \leq 1  + \int_0^t  \Vert A + M (s;r, \theta)\Vert \, \mathrm{e}^{\int_s^t  \Vert A + M (\xi;r, \theta)\Vert \, \mathrm{d}\xi}  \, \mathrm{d}s, \quad t\in [0,1]. \]

Since $\{M(t;r,\theta)\}$ is uniformly bounded, $ \Vert \mathcal{Y} (t;r,\theta) \Vert$ is uniformly bounded. Moreover, from \eqref{fdr} also $ \Vert \dot{\mathcal{Y}} (t;r,\theta) \Vert$ is uniformly bounded. Hence, we can apply Ascoli-Arzel\`a theorem to get a sub-sequence $r_k\to b$ as $k\to\infty$ and a matrix $\Phi \in \mathcal{C}^{1} ([0,1]\times\TT)$ such that
\[
\mathcal{Y} (t;r_k,\theta) \underset{k\to\infty}{\longrightarrow} \Phi(t;\theta),  \quad  \text{uniformly in} \qquad t\in [0,1] \quad  \theta\in\TT. \]

The matrices $\mathcal{Y} (t;r_k,\theta)$ satisfies
\begin{align*}
\mathcal{Y} (t;r_k,\theta) &= \mathbb{I}_2  + \int_0^t  (A + M  (s;r_k, \theta))\mathcal{Y} (s;r_k,\theta) \, \mathrm{d}s  \\ & = \mathbb{I}_2  + \int_0^t  A \mathcal{Y} (s;r_k,\theta) \, \mathrm{d}s  + \int_0^t  M  (s; r_k, \eta)\Phi(s;\theta) \, \mathrm{d}s  \\ & + \int_0^t  M  (s;r_k, \theta)(\mathcal{Y} (s;r_k,\theta)- \Phi(s;\theta)) \, \mathrm{d}s.
\end{align*}
Using the uniform convergence, and the weak* convergence of $M(t;r,\theta)$ (recall $\Phi\in  \mathcal{L}^1 ([0,1])$), we have the limit: 
\[ \lim_{k\to\infty} \mathcal{Y} (t;r_k,\theta) = \mathbb{I}_2  \,+ \int_0^t  A \Phi(s;\eta) \, \mathrm{d}s \, + \int_0^t  \tilde{M} (s; \theta)\Phi(s;\theta) \, \mathrm{d}s,  \quad t\in [0,1] \quad \forall \theta\in\TT. \]

Finally, by uniqueness we observe that this limit is the solution of system \eqref{sysweakn} and we obtain  
\[ \lim_{r\to b} \mathcal{Y} (t;r,\theta)=\Phi(t;\theta)=\tilde{\mathcal{Y}}(t;\theta) \qquad t\in [0,1] \quad \forall \theta\in\TT.\]
\end{proof}

\end{document}